\documentclass[a4paper,oneside,11pt]{article}%
\usepackage{makeidx}
\usepackage[english]{babel}
\usepackage{amsmath}
\usepackage{amsfonts}
\usepackage{amssymb}
\usepackage{stmaryrd}
\usepackage{graphicx}
\usepackage{mathrsfs}
\usepackage[colorlinks,linkcolor=red,anchorcolor=blue,citecolor=blue,urlcolor=blue]{hyperref}
\usepackage[symbol*,ragged]{footmisc}
\providecommand{\U}[1]{\protect\rule{.1in}{.1in}}
\providecommand{\U}[1]{\protect \rule{.1in}{.1in}}

\newtheorem{theorem}{Theorem}[section]

\newtheorem{lemma}[theorem]{Lemma}

\newtheorem{proposition}[theorem]{Proposition}

\newenvironment{proof}[1][Proof]{\noindent \textbf{#1.} }{\  \rule{0.5em}{0.5em}}
\numberwithin{equation}{section}

\begin{document}

\title{Waring--Goldbach Problem: One Square, \\ Four Cubes and Higher Powers }

\author{Jinjiang Li\footnotemark[1]\,\,\,\, \, \& \,\,Min Zhang\footnotemark[2] \vspace*{-4mm} \\
$\textrm{\small Department of Mathematics, China University of Mining and Technology}^{*\,\dag}$
                    \vspace*{-4mm} \\
     \small  Beijing 100083, P. R. China  }

\footnotetext[2]{Corresponding author. \\
    \quad\,\, \textit{ E-mail addresses}:
     \href{mailto:jinjiang.li.math@gmail.com}{jinjiang.li.math@gmail.com} (J. Li),
      \href{mailto:min.zhang.math@gmail.com}{min.zhang.math@gmail.com} (M. Zhang).  }

\date{}
\maketitle

{\textbf{Abstract}}: Let $\mathcal{P}_r$ denote an almost--prime with at most $r$ prime factors, counted according to multiplicity. In this paper, it is proved that, for $12\leqslant b\leqslant 35$ and for every sufficiently large odd integer $N$, the equation
\begin{equation*}
   N=x^2+p_1^3+p_2^3+p_3^3+p_4^3+p_5^4+p_6^b
\end{equation*}
is solvable with $x$ being an almost--prime $\mathcal{P}_{r(b)}$ and the other variables primes, where $r(b)$ is defined in the Theorem. This result constitutes an improvement upon that of L\"{u} and Mu.


{\textbf{MR(2010) Subject Classification}}: 11P05, 11P32, 11P55, 11N36.

\section{Introduction and main result}

Let $a,b$ and $N$ be positive integers and define $H_{a,b}(N)$ to be the number of solutions of the following Diophantine equation
\begin{equation*}
   N=x_1^2+x_2^3+x_3^3+x_4^3+x_5^3+x_6^a+x_7^b,
\end{equation*}
with all the variables $x_j$ being positive integers. In 1981, Hooley \cite{Hooley-book} obtained an asymptotic formula for $H_{3,5}(N)$. In 1991, from
Br\"{u}dern's work, Lu \cite{Lu-Ming-Gao-1991} get the asymptotic formula for $H_{3,b}(N)$. In addition, by using a sort of pruning technique, Lu \cite{Lu-Ming-Gao-1991} established the asymptotic formula for $H_{4,b}(N)\,(4\leqslant b\leqslant 6)$ and gave the lower bound estimates of the expected
order of magnitude for $H_{4,b}\,(7\leqslant b\leqslant17)$, $H_{5,b}(N)\,(5\leqslant b\leqslant9)$ and $H_{6,b}(N)\,(6\leqslant b\leqslant7)$. Motivated
by the work of Lu \cite{Lu-Ming-Gao-1991}, Dashkevich \cite{Dashkevich-1995} obtained the the asymptotic formula for $H_{6,8}(N)$.
In view of the results of Hooley, Lu and A. M. Dashkevich, it is reasonable to conjecture that, for every sufficiently large odd integer $N$ the
following equation
\begin{equation}\label{conjecture-Lv-Mu}
  N=p_1^2+p_2^3+p_3^3+p_4^3+p_5^3+p_6^a+p_7^b  \qquad (3\leqslant a\leqslant b)
\end{equation}
is solvable, where and below the letter $p$, with or without subscript, always denotes a prime number. But this conjecture is perhaps out of reach
at present. However, it is possible to replace a variable by an almost--prime. In 2016, L\"{u} and Mu \cite{Lu-Mu} employed the sieve theory and the Hardy--Littlewood method to obtain the following approximation to the conjecture (\ref{conjecture-Lv-Mu}).
\begin{theorem}[L\"{u} and Mu, 2016]\label{Lv-Mu-2016-theorem}
  Let $a$ and $b$ be positive integers such that
\begin{equation}\label{a-b-condition}
   \frac{5}{18}<\frac{1}{a}+\frac{1}{b}\leqslant\frac{1}{3}.
\end{equation}
For every sufficiently large odd integer $N$, let $\mathcal{R}_{a,b}(N)$ denote the number of solutions of the following equation
\begin{equation}
  N=x^2+p_2^3+p_3^3+p_4^3+p_5^3+p_6^a+p_7^b
\end{equation}
with $x$ being an almost--prime $\mathcal{P}_{r(a,b)}$ and the other variables primes, where $r(a,b)$ is equal to $[\frac{4}{3}(\frac{1}{a}+\frac{1}{b}-\frac{5}{18})^{-1}]$. Then we have
\begin{equation*}
  \mathcal{R}_{a,b}(N)\gg N^{\frac{1}{a}+\frac{1}{b}+\frac{13}{18}}\log^{-7}N.
\end{equation*}
\end{theorem}
Especially, if $a=4$, then there holds $12\leqslant b\leqslant35$ from the condition (\ref{a-b-condition}), and the values of $r(4,b)$ are as follows:
\begin{equation*}
  r(4,12)=24,\,\, r(4,13)=27,\,\, r(4,14)=30,\,\, r(4,15)=34,\,\, r(4,16)=38,
\end{equation*}
\begin{equation*}
  r(4,17)=42,\,\, r(4,18)=48,\,\, r(4,19)=53,\,\, r(4,20)=60,\,\, r(4,21)=67,
\end{equation*}
\begin{equation*}
  r(4,22)=75,\,\, r(4,23)=84,\,\, r(4,24)=96,\,\, r(4,25)=109,\,\, r(4,26)=124,
\end{equation*}
\begin{equation*}
  r(4,27)=144,\,\, r(4,28)=168,\,\, r(4,29)=198,\,\, r(4,30)=240,\,\, r(4,31)=297,
\end{equation*}
\begin{equation*}
  r(4,32)=384,\,\, r(4,33)=528,\,\, r(4,34)=816,\,\, r(4,35)=1680.
\end{equation*}

In this paper, we shall improve the result of L\"{u} and Mu \cite{Lu-Mu} in the cases $a=4,\,12\leqslant b\leqslant35$ and establish the following theorem.
\begin{theorem}\label{Theorem}
  For $12\leqslant b\leqslant35$, let $\mathcal{R}_b(N)$ denote the number of
solutions of the following equation
\begin{equation}
    N=x^2+p_1^3+p_2^3+p_3^3+p_4^3+p_5^4+p_6^b
\end{equation}
with $x$ being an almost--prime $\mathcal{P}_{r(b)}$ and the other variables primes. Then, for sufficiently large odd integer $N$, we have
\begin{equation*}
   \mathcal{R}_b(N)\gg N^{\frac{35}{36}+\frac{1}{b}}\log^{-7}N,
\end{equation*}
where
\begin{equation*}
  r(12)=6,\,\, r(13)=7,\,\,r(14)=7,\,\,r(15)=7,\,\,r(16)=8,\,\,r(17)=8,
\end{equation*}
\begin{equation*}
  r(18)=8,\,\,r(19)=8,\,\,r(20)=9,\,\,r(21)=9,\,\,r(22)=9,\,\,r(23)=10,
\end{equation*}
\begin{equation*}
  r(24)=10,\,\,r(25)=10,\,\, r(26)=11,\,\,r(27)=11,\,\,r(28)=11,\,\,r(29)=12,
\end{equation*}
\begin{equation*}
  r(30)=12,\,\,r(31)=13,\,\, r(32)=13,\,\,r(33)=14,\,\,r(34)=15,\,\,r(35)=17.
\end{equation*}

\end{theorem}

The proof of our result employs the Hardy--Littlewood circle method and Iwaniec's linear sieve method, from which we can give a lower bound estimate of $\mathcal{R}_b(N)$, which is stronger than that of the result of L\"{u} and Mu \cite{Lu-Mu} and leads to the refinement.

\section{Notation}
Throughout this paper, $N$ always denotes a sufficiently large odd integer; $\mathcal{P}_r$ denote an almost--prime with at most $r$ prime factors, counted according to multiplicity; $\varepsilon$ always denotes an arbitrary small positive constant, which may not be the same at different occurrences; $\gamma$ denotes Euler's constant; $f(x)\ll g(x)$ means that $f(x)=O(g(x))$; $f(x)\asymp g(x)$ means that $f(x)\ll g(x)\ll f(x)$; the letter $p$, with or without subscript, always stands for a prime number; the constants in the $O$--term and $\ll$--symbol depend at most on $\varepsilon$. As usual, $\varphi(n),\,\mu(n)$ and $\tau_k(n)$  denote Euler's function, M\"{o}bius' function and the $k$--dimensional divisor function, respectively. Especially, we write $\tau(n)=\tau_2(n)$. $p^\ell\|m$ means that $p^\ell|m$ but $p^{\ell+1}\nmid m$. We denote by $a(m)$ and $b(n)$ arithmetical functions satisfying $|a(m)|\ll1$ and $|b(n)|\ll1$; $(m,n)$ denotes the
greatest common divisor of $m$ and $n$; $e(\alpha)=e^{2\pi i\alpha}$. We always denote by $\chi$ a Dirichlet character $(\bmod q)$, and by $\chi^0$ the principal Dirichlet character $(\bmod q)$. Let
\begin{equation*}
  A=10^{100},\,\, Q_0=\log^{20A}N,\,\, Q_1=N^{\frac{19}{36}-\frac{1}{b}+50\varepsilon},\,\, Q_2=N^{\frac{17}{36}+\frac{1}{b}-50\varepsilon},\,\,
  D=N^{\frac{3}{4b}-\frac{1}{48}-51\varepsilon},
\end{equation*}
\begin{equation*}
  z=D^{\frac{1}{3}},\qquad U_k=\frac{1}{k}N^{\frac{1}{k}},\qquad U_3^*=\frac{1}{3}N^{\frac{5}{18}}, \qquad F_3(\alpha)=\sum_{U_3<n\leqslant2U_3}e(n^3\alpha),
\end{equation*}
\begin{equation*}
 F_3^*(\alpha)=\sum_{U_3^*<n\leqslant2U_3^*}e(n^3\alpha),\quad  v_k(\beta)=\int_{U_k}^{2U_k}e(\beta u^k)\mathrm{d}u,
 \quad v_3^*(\beta)=\int_{U_3^*}^{2U_3^*}e(\beta u^3)\mathrm{d}u,
\end{equation*}
\begin{equation*}
f_k(\alpha)=\sum_{U_k<p\leqslant2U_k}(\log p)e(p^k\alpha),\qquad
f_3^*(\alpha)=\sum_{U_3^*<p\leqslant2U_3^*}(\log p)e(p^3\alpha),
\end{equation*}
\begin{equation*}
  G_k(\chi,a)=\sum_{n=1}^q\chi(n)e\bigg(\frac{an^k}{q}\bigg),\quad S_k^*(q,a)=G_k(\chi^0,a),\quad S_k(q,a)=\sum_{n=1}^qe\bigg(\frac{an^k}{q}\bigg),
\end{equation*}
\begin{equation*}
  \mathcal{J}(N)=\int_{-\infty}^{+\infty}v_2(\beta)v_3^2(\beta)v_3^{*2}(\beta)v_4(\beta)v_b(\beta)e(-\beta N)\mathrm{d}\beta,
   \quad \mathcal{L}=\big\{n: U_2<n\leqslant2U_2\big\},
\end{equation*}
\begin{equation*}
 f_2(\alpha,d)=\sum_{U_2<d\ell\leqslant2U_2}e\big(\alpha(d\ell)^2\big),\qquad
 h(\alpha)=\sum_{m\leqslant D^{2/3}}a(m)\sum_{n\leqslant D^{1/3}}b(n)f_2(\alpha,mn),
\end{equation*}
\begin{equation*}
 B_d(q,N)=\sum_{\substack{a=1\\(a,q)=1}}^qS_2(q,ad^2)S_3^{*4}(q,a)S_4^{*}(q,a)S_b^{*}(q,a)e\bigg(-\frac{aN}{q}\bigg),\quad B(q,N)=B_1(q,N),
\end{equation*}
\begin{equation*}
  A_d(q,N)=\frac{B_d(q,N)}{q\varphi^6(q)}, \qquad A(q,N)=A_1(q,N),\qquad \mathfrak{S}_d(N)=\sum_{q=1}^\infty A_d(q,N),
\end{equation*}
\begin{equation*}
\mathfrak{S}(N)=\mathfrak{S}_1(N),\quad \mathfrak{P}=\prod_{2<p<z}p,\quad  \log\mathbf{U}=(\log2U_3)^2(\log2U_3^*)^2(\log2U_4)(\log2U_b),
\end{equation*}
\begin{equation*}
  \log\mathbf{W}=(\log U_3)(\log U_3^*)^2(\log U_4)(\log U_b),\quad
  g_r(\alpha)=\sum_{\substack{\ell\in\mathscr{N}_r\\ \ell p\in\mathcal{L}}}\frac{\log p}{\log\frac{U_2}{\ell}}e\big(\alpha(\ell p)^2\big),
\end{equation*}
\begin{equation*}
  \quad \mathscr{M}_r=\big\{m:U_2<m\leqslant 2U_2,m=p_1p_2\cdots p_r,\, z\leqslant p_1\leqslant p_2\leqslant\cdots\leqslant p_r\big\},
\end{equation*}
\begin{equation*}
  \quad \mathscr{N}_r=\big\{m: m=p_1p_2\cdots p_{r-1},z\leqslant p_1\leqslant p_2\leqslant\cdots\leqslant p_{r-1},\,
  p_1p_2\cdots p_{r-2}p_{r-1}^2\leqslant2U_2\big\}.
\end{equation*}

\section{Preliminary Lemmas}

In order to prove Theorem we need the following lemmas.

\begin{lemma}\label{Titchmarsh-lemma-42}
Let $F(x)$ be a real differentiable function such that $F'(x)$ is monotonic, and $F'(x)\geqslant m>0$, or $F'(x)\leqslant-m<0$, throughout
 the interval $[a,b]$. Then we have
 \begin{equation*}
     \bigg|\int_a^b e^{iF(x)}\mathrm{d}x\bigg|\leqslant\frac{4}{m}.
 \end{equation*}
\end{lemma}
\begin{proof}
   See Lemma 4.2 of Titchmarsh \cite{Titchmarsh}.
\end{proof}

\begin{lemma}\label{Titchmarsh-lemma-48}
 Let $f(x)$ be a real differentiable function in the interval $[a,b]$. If $f'(x)$ is monotonic and satisfies $|f'(x)|\leqslant\theta<1$. Then we have
\begin{equation*}
   \sum_{a<n\leqslant b}e^{2\pi if(n)}=\int_a^be^{2\pi if(x)}\mathrm{d}x+O(1).
\end{equation*}
\end{lemma}
\begin{proof}
   See Lemma 4.8 of Titchmarsh \cite{Titchmarsh}.
\end{proof}

\begin{lemma}\label{Brudern-lemma-1987}
Let $2\leqslant k_1\leqslant k_2\leqslant\cdots\leqslant k_s$ be natural numbers such that
\begin{equation*}
   \sum_{i=j+1}^s\frac{1}{k_i}\leqslant \frac{1}{k_j},\quad 1\leqslant j\leqslant s-1.
\end{equation*}
Then we have
\begin{equation*}
   \int_0^1\bigg|\prod_{i=1}^s f_{k_i}(\alpha)\bigg|^2\mathrm{d}\alpha\ll N^{\frac{1}{k_1}+\cdots+\frac{1}{k_s}+\varepsilon}.
\end{equation*}
\end{lemma}
\begin{proof}
  See Lemma 1 of Br\"{u}dern \cite{Brudern-1987}.
\end{proof}

\begin{lemma}\label{Hua-fuhe}
  For $(a,q)=1$, we have
\begin{eqnarray*}
  &  \emph{(i)} & S_{j}(q,a)\ll q^{1-\frac{1}{j}};   \\
  & \emph{(ii)} & G_{j}(\chi,a)\ll q^{\frac{1}{2}+\varepsilon}.
\end{eqnarray*}
In particular, for $(a,p)=1$, we have
\begin{eqnarray*}
  &  \emph{(iii)} &   |S_{j}(p,a)|\leqslant \big((j,p-1)-1\big)\sqrt{p};    \\
  & \emph{(iv)}   &   |S_{j}^*(p,a)|\leqslant \big((j,p-1)-1\big)\sqrt{p}+1;  \\
  & \emph{(v)}    &   S_{j}^*(p^\ell,a)=0\,\,\, \textrm{for}\,\,\, \ell\geqslant\gamma(p),\,\,\textrm{where}  \\
  &  &  \,\,\gamma(p)=\left\{
                   \begin{array}{ll}
                        \theta+2, & \textrm{if}\,\, p^\theta\|j,\, p\not=2\,\,\textrm{or}\,\, p=2,\,\theta=0,   \\
                        \theta+3, & \textrm{if}\,\, p^{\theta}\|j,\,p=2,\,\,\theta>0.
                   \end{array}
  \right.
\end{eqnarray*}
\end{lemma}
\begin{proof}
  For (i) and (iii)--(iv), see Theorem 4.2 and Lemma 4.3 of Vaughan \cite{Vaughan-book}, respectively. For (ii), see Lemma 8.5 of Hua \cite{Hua-book} or
  the Problem 14 of Chapter VI of Vinogradov \cite{Vinogradov}. For (v), see Lemma 8.3 of Hua \cite{Hua-book}.
\end{proof}

\begin{lemma}\label{f_3-mean-zh}
We have
\begin{eqnarray*}
        \emph{(i)} \quad  \int_0^1 |F_3(\alpha)F_3^{*2}(\alpha)|^2\mathrm{d}\alpha\ll N^{\frac{8}{9}+\varepsilon},
  &  &  \emph{(ii)} \quad \int_0^1 |F_3(\alpha)F_3^*(\alpha)|^4\mathrm{d}\alpha\ll N^{\frac{13}{9}},
                   \nonumber \\
        \emph{(iii)}\quad \int_0^1 |f_3(\alpha)f_3^{*2}(\alpha)|^2\mathrm{d}\alpha\ll N^{\frac{8}{9}+\varepsilon},
  &  &  \emph{(iv)} \quad \int_0^1 |f_3(\alpha)f_3^*(\alpha)|^4\mathrm{d}\alpha\ll N^{\frac{13}{9}}\log^8N.
\end{eqnarray*}
\end{lemma}
\begin{proof}
  For (i), one can see the Theorem of Vaughan \cite{Vaughan-1985}, and for (ii), one can see Lemma 2.4 of Cai \cite{Cai}. Moreover, (iii) and (iv) follow from (i) and (ii) by considering the number of solutions of the underlying Diophantine equations, respectively.
\end{proof}

\begin{lemma}\label{W(a)-Delta3}
   For $\alpha=\frac{a}{q}+\beta$, define
\begin{equation}\label{N(q,a)-def}
   \mathfrak{N}(q,a)=\bigg(\frac{a}{q}-\frac{1}{qQ_0},\frac{a}{q}+\frac{1}{qQ_0}\bigg],
\end{equation}
\begin{equation}\label{Delta_4-def}
   \Delta_4(\alpha)=f_4(\alpha)-\frac{S_4^*(q,a)}{\varphi(q)}\sum_{U_4<n\leqslant 2U_4}e(\beta n^4),
\end{equation}
\begin{equation}\label{W(alpha)-def}
   W(\alpha)=\sum_{d\leqslant D}\frac{c(d)}{dq}S_2(q,ad^2)v_2(\beta),
\end{equation}
where
\begin{equation*}
 c(d)=\sum_{\substack{d=mn\\ m\leqslant D^{2/3}\\ n\leqslant D^{1/3}}}a(m)b(n)\ll\tau(d).
\end{equation*}
Then we have
\begin{equation}\label{W(a)-Delta-mean-square}
   \sum_{1\leqslant q\leqslant Q_0}\sum_{\substack{a=-q\\(a,q)=1}}^{2q}\int_{\mathfrak{N}(q,a)}\big|W(\alpha)\Delta_4(\alpha)\big|^2\mathrm{d}\alpha
   \ll N^{\frac{1}{2}}\log^{-100A}N
\end{equation}
and
\begin{equation}\label{W(a)-mean-square}
   \sum_{1\leqslant q\leqslant Q_0}\sum_{\substack{a=-q\\(a,q)=1}}^{2q}\int_{\mathfrak{N}(q,a)}\big|W(\alpha)\big|^2\mathrm{d}\alpha
   \ll \log^{21A}N.
\end{equation}
\end{lemma}
\begin{proof}
 For (\ref{W(a)-Delta-mean-square}) and (\ref{W(a)-mean-square}), one can refer to Lemma 2.4 and Lemma 2.5 of Li and Cai \cite{Li-Cai}, respectively.
\end{proof}

\begin{lemma}\label{V_6(alpha)-mean-square}
  For $\alpha=\frac{a}{q}+\beta$, define
   \begin{equation}\label{V_k-def}
      V_k(\alpha)=\frac{S_k^*(q,a)}{\varphi(q)}v_k(\beta),
   \end{equation}
 Then we have
 \begin{equation}\label{V-6-mean-square}
   \sum_{1\leqslant q\leqslant Q_0}\sum_{\substack{a=-q\\(a,q)=1}}^{2q}\int_{\mathfrak{N}(q,a)}\big|V_4(\alpha)\big|^2\mathrm{d}\alpha
   \ll N^{-\frac{1}{2}}\log^{21A}N,
\end{equation}
where $\mathfrak{N}(q,a)$ is defined by (\ref{N(q,a)-def}).
\end{lemma}
\begin{proof}
  See (2.12) of Li and Cai \cite{Li-Cai}.
\end{proof}

For $(a,q)=1,\,1\leqslant a\leqslant q\leqslant Q_2$, set
\begin{equation*}
  \mathfrak{M}(q,a)=\bigg(\frac{a}{q}-\frac{1}{qQ_2},\frac{a}{q}+\frac{1}{qQ_2}\bigg],\qquad\quad
  \mathfrak{M}=\bigcup_{1\leqslant q\leqslant Q_0^5}\bigcup_{\substack{1\leqslant a\leqslant q\\(a,q)=1}}\mathfrak{M}(q,a),
\end{equation*}
\begin{equation*}
  \mathfrak{M}_0(q,a)=\bigg(\frac{a}{q}-\frac{Q_0}{N},\frac{a}{q}+\frac{Q_0}{N}\bigg],\qquad\quad
  \mathfrak{M}_0=\bigcup_{1\leqslant q\leqslant Q_0^5}\bigcup_{\substack{1\leqslant a\leqslant q\\(a,q)=1}}\mathfrak{M}_0(q,a),
\end{equation*}
\begin{equation*}
  \mathfrak{I}_0=\bigg(-\frac{1}{Q_2},1-\frac{1}{Q_2}\bigg],\quad\quad \qquad \mathfrak{m}_0=\mathfrak{M}\setminus \mathfrak{M}_0,
\end{equation*}
\begin{equation*}
  \mathfrak{m}_1=\bigcup_{Q_0^5<q\leqslant Q_1}\bigcup_{\substack{1\leqslant a\leqslant q\\(a,q)=1}}\mathfrak{M}(q,a),\quad\quad\qquad
  \mathfrak{m}_2=\mathfrak{I}_0\setminus (\mathfrak{M}\cup \mathfrak{m}_1).
\end{equation*}
Then we get the Farey dissection
\begin{equation}\label{Farey-dissection}
  \mathfrak{I}_0=\mathfrak{M}_0\cup\mathfrak{m}_0\cup\mathfrak{m}_1\cup\mathfrak{m}_2.
\end{equation}

\begin{lemma}\label{f_k-tihuan}
 For $\alpha=\frac{a}{q}+\beta$, define
\begin{equation*}
      V_3^*(\alpha)=\frac{S_3^*(q,a)}{\varphi(q)}v_3^*(\beta).
   \end{equation*}
Then $\alpha=\frac{a}{q}+\beta\in\mathfrak{M}_0$, we have
   \begin{equation}\label{f_3=V_3+O}
      f_k(\alpha)=V_k(\alpha)+O\big(U_k\exp(-\log^{1/3}N)\big),
   \end{equation}
   \begin{equation}\label{f_3^*=W_3+O}
      f_3^*(\alpha)=V_3^*(\alpha)+O\big(U_3^*\exp(-\log^{1/3}N)\big),
   \end{equation}
   \begin{equation}\label{g_r=c_rV_2+O}
     g_r(\alpha)=\frac{c_r(b)V_2(\alpha)}{\log U_2}+O\big(U_2\exp(-\log^{1/3}N)\big),
   \end{equation}
where $V_k(\alpha)$ is defined (\ref{V_k-def}), and
\begin{align}\label{c_r-def}
   c_r(b) = & (1+O(\varepsilon)) \nonumber  \\
   & \times \int_{r-1}^{\frac{73b-36}{36-b}}\frac{\mathrm{d}t_1}{t_1}\int_{r-2}^{t_1-1}\frac{\mathrm{d}t_2}{t_2}\cdots
         \int_{3}^{t_{r-4}-1}\frac{\mathrm{d}t_{r-3}}{t_{r-3}}\int_{2}^{t_{r-3}-1}\frac{\log (t_{r-2}-1)}{t_{r-2}}\mathrm{d}t_{r-2}.
\end{align}
\end{lemma}
\begin{proof}
  By Siegel--Walfisz theorem and partial summation, we obtain
\begin{eqnarray}
  g_r(\alpha) & = & \sum_{\substack{\ell\in\mathscr{N}_r \\ \ell p\in\mathcal{L} }}
        e\bigg(\Big(\frac{a}{q}+\beta\Big)(\ell p)^2\bigg)\frac{\log p}{\log\frac{U_2}{\ell}}
              \nonumber \\
    & = & \sum_{\substack{h=1\\(h,q)=1}}^q e\bigg(\frac{ah^2}{q}\bigg)\sum_{\ell\in\mathscr{N}_r}\frac{1}{\log\frac{U_2}{\ell}}
          \sum_{\substack{\frac{U_2}{\ell}<p\leqslant\frac{2U_2}{\ell} \\ \ell p\equiv h \!\!\!\!\! \pmod q }} (\log p)e\big(\beta(\ell p)^2\big)
             \nonumber \\
    & = & \sum_{\substack{h=1\\(h,q)=1}}^q e\bigg(\frac{ah^2}{q}\bigg)\sum_{\ell\in\mathscr{N}_r}\frac{1}{\log\frac{U_2}{\ell}}
          \int_{\frac{U_2}{\ell}}^{\frac{2U_2}{\ell}} e\big(\beta(\ell\nu)^2\big) \mathrm{d}
          \Bigg(\sum_{\substack{p\leqslant\nu \\ p\equiv h\bar{\ell^{-1}}\!\!\!\!\! \pmod q }} \log p\Bigg)
               \nonumber \\
    & = & \frac{S_2^*(q,a)}{\varphi(q)}v_2(\beta)\sum_{\ell\in\mathscr{N}_r}\frac{1}{\ell\log\frac{U_2}{\ell}}
           +O\big(U_2\exp(-\log^{1/3}N)\big)
               \nonumber \\
    & = & \frac{c_r(b)V_2(\alpha)}{\log U_2}+O\big(U_2\exp(-\log^{1/3}N)\big).
\end{eqnarray}
This completes the proof of (\ref{g_r=c_rV_2+O}). Also, (\ref{f_3=V_3+O}) and (\ref{f_3^*=W_3+O}) can be proved in similar but simpler processes.
\end{proof}

\begin{lemma}\label{h(a)-upper-m2}
For $\alpha\in\mathfrak{m}_2$, we have
  \begin{equation*}
     h(\alpha)\ll N^{\frac{17}{72}+\frac{1}{2b}-24\varepsilon}.
  \end{equation*}
\end{lemma}
\begin{proof}
  By the estimate (4.5) of Lemma 4.2 in Br\"{u}dern and Kawada \cite{Brudern-Kawada}, we deduce that
\begin{align*}
   h(\alpha) \ll & \quad \frac{N^{\frac{1}{2}}\tau^2(q)\log^2N}{(q+N|q\alpha-a|)^{1/2}}+N^{\frac{1}{4}+\varepsilon}D^{\frac{2}{3}}  \\
   \ll & \quad N^{\frac{1}{2}+\varepsilon}Q_1^{-\frac{1}{2}}+N^{\frac{1}{4}+\varepsilon}D^{\frac{2}{3}}\ll N^{\frac{17}{72}+\frac{1}{2b}-24\varepsilon}.
\end{align*}
This completes the proof of Lemma \ref{h(a)-upper-m2}.
\end{proof}

\section{Mean Value Theorems}

In this section, we shall prove the mean value theorems for the proof of  Theorem \ref{Theorem}.

\begin{proposition}\label{mean-value-theorem-1}
For $12\leqslant b\leqslant35$, define
\begin{equation*}
  J(N,d)=\sum_{\substack{m^2+p_1^3+p_2^3+p_3^3+p_4^3+p_5^4+p_6^b=N\\ m\in\mathcal{L},\quad m\equiv 0 \!\!\!\pmod d \\
  U_3<p_1,\,p_2\leqslant 2U_3,\, U_3^*<p_3,p_4\leqslant2U_3^* \\ U_4<p_5\leqslant 2U_4,\, U_b<p_6\leqslant 2U_b}}
  \prod_{j=1}^6\log p_j.
\end{equation*}
Then we have
\begin{equation*}
  \sum_{m\leqslant D^{2/3}}a(m)\sum_{n\leqslant D^{1/3}}b(n)\bigg(J(N,mn)-\frac{\mathfrak{S}_{mn}(N)}{mn}\mathcal{J}(N)\bigg)
  \ll N^{\frac{35}{36}+\frac{1}{b}}\log^{-A}N.
\end{equation*}
\end{proposition}

\begin{proof}
   Let
\begin{equation*}
  K(\alpha)=h(\alpha)f_3^2(\alpha)f_3^{*2}(\alpha)f_4(\alpha)f_b(\alpha)e(-N\alpha).
\end{equation*}
By the Farey dissection (\ref{Farey-dissection}), we have
\begin{align}\label{junzhi-fenjie}
  & \quad \sum_{m\leqslant D^{2/3}}a(m)\sum_{n\leqslant D^{1/3}}b(n)J(N,mn)
              \nonumber \\
  = & \quad \int_{\mathfrak{I}_0}K(\alpha)\mathrm{d}\alpha=
  \bigg(\int_{\mathfrak{M}_0}+\int_{\mathfrak{m}_0}+\int_{\mathfrak{m}_1}+\int_{\mathfrak{m}_2}\bigg)K(\alpha)\mathrm{d}\alpha.
\end{align}
From Cauchy's inequality, Lemma \ref{Brudern-lemma-1987} and (iii) of Lemma \ref{f_3-mean-zh}, we obtain
\begin{align*}
         & \,\, \int_0^1|f_3^2(\alpha)f_3^{*2}(\alpha)f_4(\alpha)f_b(\alpha)|\mathrm{d}\alpha     \nonumber \\
\end{align*}
\begin{align}\label{f_3^2-f_3^*2-f_6-f_7}
     \ll & \,\,  \bigg(\int_0^1|f_3(\alpha)f_4(\alpha)f_b(\alpha)|^2\mathrm{d}\alpha\bigg)^{\frac{1}{2}}
             \bigg(\int_0^1\big|f_3(\alpha)f_3^{*2}(\alpha)\big|^2\mathrm{d}\alpha\bigg)^{\frac{1}{2}}    \nonumber \\
     \ll & \,\,  (N^{\frac{7}{12}+\frac{1}{b}+\varepsilon})^{1/2}(N^{\frac{8}{9}+\varepsilon})^{1/2} \ll N^{\frac{53}{72}+\frac{1}{2b}+\varepsilon}.
\end{align}
By Lemma \ref{h(a)-upper-m2} and (\ref{f_3^2-f_3^*2-f_6-f_7}), we get
\begin{align}\label{K(a)-upper-m2}
              \int_{\mathfrak{m}_2}K(\alpha)\mathrm{d}\alpha
  \ll & \,\, \sup_{\alpha\in\mathfrak{m}_2}|h(\alpha)|\int_0^1|f_3^2(\alpha)f_3^{*2}(\alpha)f_4(\alpha)f_b(\alpha)|\mathrm{d}\alpha   \nonumber \\
  \ll & \,\, N^{\frac{17}{72}+\frac{1}{2b}-24\varepsilon}\cdot N^{\frac{53}{72}+\frac{1}{2b}+\varepsilon}\ll N^{\frac{35}{36}+\frac{1}{b}-\varepsilon}.
\end{align}
 From Theorem 4.1 of Vaughan \cite{Vaughan-book}, for $\alpha\in\mathfrak{m}_1$, we have
\begin{equation} \label{h(a)=W(a)+error}
    h(\alpha)=W(\alpha)+O(DQ_1^{\frac{1}{2}+\varepsilon})=W(\alpha)+O(N^{\frac{35}{144}+\frac{1}{4b}-25\varepsilon}),
\end{equation}
 where $W(\alpha)$ is defined by (\ref{W(alpha)-def}). Define
\begin{equation}\label{K_1-def}
  K_1(\alpha)=W(\alpha)f_3^2(\alpha)f_3^{*2}(\alpha)f_4(\alpha)f_b(\alpha)e(-N\alpha).
\end{equation}
Then, by (\ref{f_3^2-f_3^*2-f_6-f_7}) and (\ref{h(a)=W(a)+error}), we have
\begin{equation}\label{K(a)=K_1(a)+error}
   \int_{\mathfrak{m}_1}K(\alpha)\mathrm{d}\alpha=\int_{\mathfrak{m}_1}K_1(\alpha)\mathrm{d}\alpha+O(N^{\frac{47}{48}+\frac{3}{4b}-20\varepsilon}).
\end{equation}
Let
\begin{equation*}
   \mathfrak{N}_0(q,a)=\bigg(\frac{a}{q}-\frac{1}{N^{709/945}},\frac{a}{q}+\frac{1}{N^{709/945}}\bigg],\qquad
   \mathfrak{N}_0=\bigcup_{1\leqslant q\leqslant Q_0}\bigcup_{\substack{a=-q\\(a,q)=1}}^{2q}\mathfrak{N}_0(q,a),
\end{equation*}
\begin{equation*}
   \mathfrak{N}_1(q,a)=\mathfrak{N}(q,a)\setminus\mathfrak{N}_0(q,a),\qquad\quad
    \mathfrak{N}_1=\bigcup_{1\leqslant q\leqslant Q_0}\bigcup_{\substack{a=-q\\(a,q)=1}}^{2q}\mathfrak{N}_1(q,a),
\end{equation*}
\begin{equation*}
    \mathfrak{N}=\bigcup_{1\leqslant q\leqslant Q_0}\bigcup_{\substack{a=-q\\(a,q)=1}}^{2q}\mathfrak{N}(q,a),
\end{equation*}
where $\mathfrak{N}(q,a)$ is defined by (\ref{N(q,a)-def}). Then we have $\mathfrak{m}_1\subset\mathfrak{I}_0\subset\mathfrak{N}$.
By the rational approximation theorem of Dirichlet, we get
\begin{align}\label{K_1-upper=1+2}
            \int_{\mathfrak{m}_1}K_1(\alpha)\mathrm{d}\alpha
  \ll &  \int_{\mathfrak{m}_1\cap\mathfrak{N}_0}|K_1(\alpha)|\mathrm{d}\alpha+\int_{\mathfrak{m}_1\cap\mathfrak{N}_1}|K_1(\alpha)|\mathrm{d}\alpha
             \nonumber   \\
  \ll &   \sum_{1\leqslant q\leqslant Q_0}\sum_{\substack{a=-q\\(a,q)=1}}^{2q}\int_{\mathfrak{m}_1\cap\mathfrak{N}_0(q,a)}|K_1(\alpha)|\mathrm{d}\alpha
            \nonumber \\
    &  +\sum_{1\leqslant q\leqslant Q_0}\sum_{\substack{a=-q\\(a,q)=1}}^{2q}\int_{\mathfrak{m}_1\cap\mathfrak{N}_1(q,a)}|K_1(\alpha)|\mathrm{d}\alpha .
\end{align}
By Lemma \ref{Titchmarsh-lemma-42}, we have
\begin{equation*}
    v_k(\beta)\ll \frac{U_k}{1+|\beta|N}.
\end{equation*}
From the trivial inequality $(q,d^2)\leqslant(q,d)^2$ and above estimate, we have
\begin{eqnarray} \label{W(a)-upper}
   |W(\alpha)| & \ll & \sum_{d\leqslant D}\frac{\tau(d)}{d}(q,d^2)^{1/2}q^{-1/2}|v_2(\beta)|
                         \nonumber \\
   & \ll & \tau_3(q)q^{-1/2}|v_2(\beta)|\log^2N \ll \frac{\tau_3(q)U_2\log^2N}{q^{1/2}(1+|\beta|N)}.
\end{eqnarray}
 Thus, for $\alpha\in\mathfrak{N}_1(q,a)$, we get
\begin{equation}
 W(\alpha)\ll N^{\frac{473}{1890}}\log^2N,
\end{equation}
from which and (\ref{f_3^2-f_3^*2-f_6-f_7}) we have
\begin{eqnarray}\label{K_1-upper-2}
    &  & \sum_{1\leqslant q\leqslant Q_0}\sum_{\substack{a=-q\\(a,q)=1}}^{2q}\int_{\mathfrak{m}_1\cap\mathfrak{N}_1(q,a)}|K_1(\alpha)|\mathrm{d}\alpha
             \nonumber \\
    & \ll & N^{\frac{473}{1890}}\log^2N\cdot\int_0^1\big|f_3^2(\alpha)f_3^{*2}(\alpha)f_6(\alpha)f_7(\alpha)\big|\mathrm{d}\alpha\ll N^{\frac{35}{36}+\frac{1}{b}-\varepsilon}.
\end{eqnarray}
 By Lemma \ref{Titchmarsh-lemma-48}, we derive
\begin{equation*}
   f_4(\alpha)=\Delta_4(\alpha)+V_4(\alpha)+O(1).
\end{equation*}
Therefore, we have
\begin{eqnarray}\label{m_1-cap-N0-fenjie}
    &  & \sum_{1\leqslant q\leqslant Q_0}\sum_{\substack{a=-q\\(a,q)=1}}^{2q}\int_{\mathfrak{m}_1\cap\mathfrak{N}_0(q,a)}|K_1(\alpha)|\mathrm{d}\alpha
             \nonumber \\
    & \ll & \sum_{1\leqslant q\leqslant Q_0}\sum_{\substack{a=-q\\(a,q)=1}}^{2q}\int_{\mathfrak{m}_1\cap\mathfrak{N}_0(q,a)}
            \big|W(\alpha)\Delta_4(\alpha)f_3^2(\alpha)f_3^{*2}(\alpha)f_b(\alpha)\big|\mathrm{d}\alpha
                \nonumber \\
    &     & +\sum_{1\leqslant q\leqslant Q_0}\sum_{\substack{a=-q\\(a,q)=1}}^{2q}\int_{\mathfrak{m}_1\cap\mathfrak{N}_0(q,a)}
            \big|W(\alpha)V_4(\alpha)f_3^2(\alpha)f_3^{*2}(\alpha)f_b(\alpha)\big|\mathrm{d}\alpha
                 \nonumber \\
    & & + \sum_{1\leqslant q\leqslant Q_0}\sum_{\substack{a=-q\\(a,q)=1}}^{2q}\int_{\mathfrak{m}_1\cap\mathfrak{N}_0(q,a)}
                  \big|W(\alpha)f_3^2(\alpha)f_3^{*2}(\alpha)f_b(\alpha)\big|\mathrm{d}\alpha
                       \nonumber \\
    & =: & \mathcal{I}_1+\mathcal{I}_2+\mathcal{I}_3,
\end{eqnarray}
where $\Delta_4(\alpha)$ and $V_4(\alpha)$ are defined by (\ref{Delta_4-def}) and (\ref{V_k-def}), respectively.

It follows from Cauchy's inequality, (iv) of Lemma \ref{f_3-mean-zh} and (\ref{W(a)-Delta-mean-square}) that
\begin{align}\label{I_1-upper}
   \mathcal{I}_1  \ll & \sup_{\alpha\in\mathfrak{N}_0}|f_b(\alpha)|
                        \Bigg(\sum_{1\leqslant q\leqslant Q_0}\sum_{\substack{a=-q\\(a,q)=1}}^{2q}\int_{\mathfrak{N}(q,a)}
                        \big|W(\alpha)\Delta_4(\alpha)\big|^2\mathrm{d}\alpha\Bigg)^{\frac{1}{2}}
                        \bigg(\int_0^1\big|f_3(\alpha)f_3^*(\alpha)\big|^4\mathrm{d}\alpha\bigg)^{\frac{1}{2}}
                            \nonumber   \\
        \ll & \,\, N^{\frac{1}{b}}\big(N^{\frac{1}{2}}\log^{-100A}N\big)^{1/2}\big(N^{\frac{13}{9}}\log^8N\big)^{1/2}
                   \ll N^{\frac{35}{36}+\frac{1}{b}}\log^{-40A}N.
\end{align}
 From (\ref{W(a)-upper}), we know that, for $\alpha\in\mathfrak{m}_1$, there holds
\begin{equation}\label{W(a)-upper-m_1}
  \sup_{\alpha\in\mathfrak{m}_1} |W(\alpha)|\ll N^{\frac{1}{2}}\log^{-30A}N.
\end{equation}
Therefore, by Cauchy's inequality, (\ref{V-6-mean-square}), (\ref{W(a)-upper-m_1}) and (iv) of Lemma \ref{f_3-mean-zh}, we obtain
\begin{align}\label{I_2-upper}
   \mathcal{I}_2 \ll & \bigg(\sup_{\alpha\in\mathfrak{N}_0}|f_b(\alpha)|\bigg)\bigg(  \sup_{\alpha\in\mathfrak{m}_1} |W(\alpha)| \bigg)\cdot
              \Bigg(\sum_{1\leqslant q\leqslant Q_0}\sum_{\substack{a=-q\\(a,q)=1}}^{2q}\int_{\mathfrak{N}(q,a)}
                   \big|V_4(\alpha)\big|^2\mathrm{d}\alpha\Bigg)^{1/2}
                     \nonumber  \\
           & \,\times  \bigg(\int_0^1\big|f_3(\alpha)f_3^*(\alpha)\big|^4\mathrm{d}\alpha\bigg)^{1/2}
                     \nonumber   \\
        \ll & \,  N^{\frac{1}{b}}\cdot N^{\frac{1}{2}}\log^{-30A}N \cdot \big(N^{-\frac{1}{2}}\log^{21A}N\big)^{\frac{1}{2}} \cdot
                   \big(N^{\frac{13}{9}}\log^8N\big)^{\frac{1}{2}}
                      \nonumber   \\
        \ll & \, N^{\frac{35}{36}+\frac{1}{b}}\log^{-10A}N.
\end{align}
From Cauchy's inequality, (\ref{W(a)-mean-square}), and (iv) of Lemma \ref{f_3-mean-zh}, we derive that
\begin{align}\label{I_3-upper}
   \mathcal{I}_3  \ll &  \,\sup_{\alpha\in\mathfrak{N}_0}|f_b(\alpha)|
                         \Bigg(\sum_{1\leqslant q\leqslant Q_0}\sum_{\substack{a=-q\\(a,q)=1}}^{2q}\int_{\mathfrak{N}(q,a)}
                          \big|W(\alpha)\big|^2\mathrm{d}\alpha\Bigg)^{1/2}
                          \bigg(\int_0^1\big|f_3(\alpha)f_3^*(\alpha)\big|^4\mathrm{d}\alpha\bigg)^{1/2}
                      \nonumber   \\
        \ll &   \,  N^{\frac{1}{b}}\cdot(\log^{21A}N)^{\frac{1}{2}}(N^{\frac{13}{9}}\log^{8}N)^{\frac{1}{2}}\ll N^{\frac{13}{18}+\frac{1}{b}}\log^{20A}N
                    \ll N^{\frac{35}{36}+\frac{1}{b}}\log^{-10A}N.
\end{align}
Combining (\ref{m_1-cap-N0-fenjie}), (\ref{I_1-upper}), (\ref{I_2-upper}) and (\ref{I_3-upper}), we can deduce that
\begin{equation}\label{K_1-upper-1}
\sum_{1\leqslant q\leqslant Q_0}\sum_{\substack{a=-q\\(a,q)=1}}^{2q}\int_{\mathfrak{m}_1\cap\mathfrak{N}_0(q,a)}|K_1(\alpha)|\mathrm{d}\alpha
\ll N^{\frac{35}{36}+\frac{1}{b}}\log^{-10A}N.
\end{equation}
From (\ref{K(a)=K_1(a)+error}), (\ref{K_1-upper=1+2}), (\ref{K_1-upper-2}) and (\ref{K_1-upper-1}) we conclude that
\begin{equation}\label{K(a)-upper-m1}
  \int_{\mathfrak{m}_1}K(\alpha)\mathrm{d}\alpha\ll N^{\frac{35}{36}+\frac{1}{b}}\log^{-10A}N.
\end{equation}
Similarly, we obtain
\begin{equation}\label{K(a)-upper-m0}
  \int_{\mathfrak{m}_0}K(\alpha)\mathrm{d}\alpha\ll N^{\frac{35}{36}+\frac{1}{b}}\log^{-10A}N.
\end{equation}

For $\alpha\in\mathfrak{M}_0$, define
\begin{equation*}
  K_0(\alpha)=W(\alpha)V_3^2(\alpha)V_3^{*2}(\alpha)V_4(\alpha)V_b(\alpha)e(-N\alpha).
\end{equation*}
Noticing that (\ref{h(a)=W(a)+error}) still holds for $\alpha\in\mathfrak{M}_0$, it follows
from (\ref{f_3=V_3+O}), (\ref{f_3^*=W_3+O}) and (\ref{h(a)=W(a)+error}) that
\begin{equation*}
   K(\alpha)-K_0(\alpha)\ll N^{\frac{71}{36}+\frac{1}{b}}\exp\big(-\log^{1/4}N\big).
\end{equation*}
By the above estimate, we derive that
\begin{equation}\label{K(a)=K_0(a)+error}
   \int_{\mathfrak{M}_0}K(\alpha)\mathrm{d}\alpha = \int_{\mathfrak{M}_0}K_0(\alpha)\mathrm{d}\alpha +O\big( N^{\frac{35}{36}+\frac{1}{b}}\log^{-A}N\big).
\end{equation}
By the well--known standard technique  in the Hardy--Littlewood method, we deduce that
\begin{equation}\label{K_0=junzhi}
    \int_{\mathfrak{M}_0}K_0(\alpha)\mathrm{d}\alpha =
    \sum_{m\leqslant D^{2/3}}a(m)\sum_{n\leqslant D^{1/3}}b(n)\frac{\mathfrak{S}_{mn}(N)}{mn}\mathcal{J}(N)+O\big( N^{\frac{35}{36}+\frac{1}{b}}\log^{-A}N\big),
\end{equation}
and
\begin{equation}\label{singular-int}
    \mathcal{J}(N)\asymp N^{\frac{35}{36}+\frac{1}{b}}.
\end{equation}
From (\ref{junzhi-fenjie}), (\ref{K(a)-upper-m2}), (\ref{K(a)-upper-m1})--(\ref{singular-int}) , the
result of Proposition \ref{mean-value-theorem-1} follows.
\end{proof}

In a similar way, we have

\begin{proposition}\label{mean-value-theorem-2}
For $12\leqslant b\leqslant35$, define
\begin{equation*}
  J_r(N,d)=\sum_{\substack{(\ell p)^2+m^3+p_2^3+p_3^3+p_4^3+p_5^4+p_6^b=N \\ \ell p\in\mathcal{L}, \,\,\, \ell\in\mathscr{N}_r,  \,\,\,
   m\equiv0 (\!\bmod d) \\
    U_3<p_2\leqslant2U_3,\, U_3^*<p_3,\,p_4\leqslant 2U_3^* \\ U_4<p_5\leqslant 2U_4,\, U_b<p_6\leqslant2U_b}}
  \left(\frac{\log p}{\log \frac{U_2}{\ell}}\prod_{j=2}^6\log p_j\right).
\end{equation*}
Then we have
\begin{equation*}
  \sum_{m\leqslant D^{2/3}}a(m)\sum_{n\leqslant D^{1/3}}b(n)\bigg(J_r(N,mn)-\frac{c_r(b)\mathfrak{S}_{mn}(N)}{mn\log U_2}\mathcal{J}(N)\bigg)
    \ll N^{\frac{35}{36}+\frac{1}{b}}\log^{-A}N,
\end{equation*}
where $c_r(b)$ is defined by (\ref{c_r-def}).
\end{proposition}

\section{On the function $\omega(d)$}

In this section, we shall investigate the function $\omega(d)$ which is defined in (\ref{omega(d)-def}) and required in the proof of the Theorem \ref{Theorem}.

\begin{lemma}\label{congruence-lemma}
   Let $\mathfrak{K}(q,N)$ and $\mathfrak{L}(q,N)$ denote the number of solutions of the following congruences
\begin{equation*}
   u_1^3+u_2^3+u_3^3+u_4^3+u_5^4+p_6^b\equiv N (\bmod q), \,\,\, 1\leqslant u_j\leqslant q,\,\,\, (u_j,q)=1, \,\,\,1\leqslant j\leqslant6,
\end{equation*}
 and
\begin{equation*}
     x^2+u_1^3+u_2^3+u_3^3+u_4^3+u_5^4+p_6^b\equiv N (\bmod q),\,\,\, 1\leqslant x,u_j\leqslant q, \,\,\, (u_j,q)=1, \,\,\, 1\leqslant j\leqslant6,
\end{equation*}
 respectively. Then we have $\mathfrak{L}(p,N)>\mathfrak{K}(p,N)$. Moreover, there holds
\begin{equation}\label{L(p,N)-asymptotic}
    \mathfrak{L}(p,N)=p^6+O(p^5),
\end{equation}
\begin{equation}\label{K(p,N)-asymptotic}
    \mathfrak{K}(p,N)=p^5+O(p^4).
\end{equation}
\end{lemma}
\begin{proof}
   Let $\mathfrak{L}^*(q,N)$ denote the number of solutions to the following congruence
\begin{equation*}
     x^2+u_1^3+u_2^3+u_3^3+u_4^3+u_5^4+p_6^b\equiv N (\bmod q),\,\,\, 1\leqslant x,u_j\leqslant q, \,\,\, (xu_j,q)=1, \,\,\, 1\leqslant j\leqslant6.
\end{equation*}
Then we have
\begin{align}\label{p-times-L^*(p,N)}
          p\cdot \mathfrak{L}^*(p,N)
     = & \,\, \sum_{a=1}^pS_2^*(p,a)S_3^{*4}(p,a)S_4^*(p,a)S_b^*(p,a)e\bigg(-\frac{aN}{p}\bigg)
                 \nonumber    \\
     = & \,\, (p-1)^7+E_p,
\end{align}
where
\begin{equation*}
  E_p=\sum_{a=1}^{p-1}S_2^*(p,a)S_3^{*4}(p,a)S_4^*(p,a)S_b^*(p,a)e\bigg(-\frac{aN}{p}\bigg).
\end{equation*}
By (iv) of Lemma \ref{Hua-fuhe}, we obtain
\begin{equation}\label{E_p-upper}
  |E_p|\leqslant (p-1)^2(\sqrt{p}+1)(2\sqrt{p}+1)^4(3\sqrt{p}+1).
\end{equation}
It is easy to verify that $|E_p|<(p-1)^7$ for $p\geqslant13$, hence we have $\mathfrak{L}^*(p,N)>0$ for $p\geqslant 13$.
In addition, for $p=2,3,5,7,11$, we can check one by one directly by hand that $\mathfrak{L}^*(p,N)>0$. Therefore, we have
$\mathfrak{L}^*(p,N)>0$ for every prime $p$, and
\begin{equation}\label{L(p,N)=L^*(p,N)+K(p,N)}
    \mathfrak{L}(p,N)=\mathfrak{L}^*(p,N)+\mathfrak{K}(p,N)>\mathfrak{K}(p,N).
\end{equation}
From (\ref{p-times-L^*(p,N)}) and (\ref{E_p-upper}), we deduce that
\begin{equation}
    \mathfrak{L}^*(p,N)=p^6+O(p^5).
\end{equation}
By similar arguments that lead to (\ref{p-times-L^*(p,N)}) and (\ref{E_p-upper}), we have
\begin{equation}\label{K(p,N)=p^5+O(p^4)}
    \mathfrak{K}(p,N)=p^5+O(p^4).
\end{equation}
Combining (\ref{L(p,N)=L^*(p,N)+K(p,N)})--(\ref{K(p,N)=p^5+O(p^4)}), we get (\ref{L(p,N)-asymptotic}).
 This completes the proof of Lemma \ref{congruence-lemma}.
\end{proof}

\begin{lemma}\label{S(N)-convergence}
   The series $\mathfrak{S}(N)$ is convergent and satisfying $\mathfrak{S}(N)>0$.
\end{lemma}
\begin{proof}
 From (i) and (ii) of Lemma \ref{Hua-fuhe}, we get
\begin{equation*}
  |A(q,N)|\ll\frac{|B(q,N)|}{q\varphi^6(q)}\ll \frac{q^{5/2+6\varepsilon}}{\varphi^5(q)}\ll\frac{q^{5/2+6\varepsilon}(\log\log q)^5}{q^5}
  \ll \frac{1}{q^2}.
\end{equation*}
Thus, the series
\begin{equation*}
  \mathfrak{S}(N)=\sum_{q=1}^\infty A(q,N)
\end{equation*}
converges absolutely. Noting that $A(q,N)$ is multiplicative in $q$ and by (v) of Lemma \ref{Hua-fuhe}, we have
\begin{equation}\label{S(N)-prod-extension}
   \mathfrak{S}(N)=\prod_p \big(1+A(p,N)\big).
\end{equation}
From (iii) and (iv), we know that, for $p\geqslant19$, there holds
\begin{equation*}
  |A(p,N)|\leqslant\frac{(p-1)^2\sqrt{p}(2\sqrt{p}+1)^4(3\sqrt{p}+1)}{p(p-1)^6}\leqslant\frac{100}{p^2}.
\end{equation*}
Therefore, there holds
\begin{equation}\label{S(N)-wei-lower}
  \prod_{p\geqslant19}\big(1+A(p,N)\big)\geqslant\prod_{p\geqslant19}\bigg(1-\frac{100}{p^2}\bigg)\geqslant c_1>0.
\end{equation}
On the other hand, it is easy to see that
\begin{equation}\label{1+A(p,N)=jieshu}
  1+A(p,N)=\frac{\mathfrak{L}(p,N)}{(p-1)^6},
\end{equation}
from which and (\ref{L(p,N)=L^*(p,N)+K(p,N)}), we have $1+A(p,N)>0$.
Therefore, there holds
\begin{equation}\label{S(N)-shou-lower}
  \prod_{p<19}\big(1+A(p,N)\big)\geqslant c_2>0.
\end{equation}
Finally, from (\ref{S(N)-prod-extension})--(\ref{S(N)-shou-lower}), we conclude that $\mathfrak{S}(N)>0$. This completes the proof of Lemma \ref{S(N)-convergence}.
\end{proof}

In view of Lemma \ref{S(N)-convergence}, we define
\begin{equation}\label{omega(d)-def}
   \omega(d)=\frac{\mathfrak{S}_d(N)}{\mathfrak{S}(N)}.
\end{equation}
Similar to (\ref{S(N)-prod-extension}), we have
\begin{equation}\label{singular(S)_d-explicit}
  \mathfrak{S}_d(N)=\prod_p\big(1+A_d(p,N)\big)=\prod_{p\nmid d}\big(1+A_d(p,N)\big) \prod_{p|d}\big(1+A_d(p,N)\big).
\end{equation}
If $(d,q)=1$, then we have $S_k(q,ad^k)=S_k(q,a)$. Moreover, if $p|d$, then we get $A_d(p,N)=A_p(p,N)$. Therefore, it follows
from (\ref{S(N)-prod-extension}), (\ref{omega(d)-def}) and (\ref{singular(S)_d-explicit}) that
\begin{equation}\label{omega(p)-fd}
   \omega(p)=\frac{1+A_p(p,N)}{1+A(p,N)}, \qquad \omega(d)=\prod_{p|d}\omega(p).
\end{equation}
Also, it is easy to show that
\begin{equation}\label{omega(p)-yz-1}
  1+A_p(p,N)=\frac{p}{(p-1)^6}\mathfrak{K}(p,N).
\end{equation}
From (\ref{1+A(p,N)=jieshu}), (\ref{omega(p)-fd}) and (\ref{omega(p)-yz-1}), we deduce that
\begin{equation}\label{omega(p)-solution}
  \omega(p)=\frac{p\cdot \mathfrak{K}(p,N)}{\mathfrak{L}(p,N)}.
\end{equation}
According to (\ref{L(p,N)-asymptotic}), (\ref{K(p,N)-asymptotic}), (\ref{omega(p)-fd}) and (\ref{omega(p)-solution}), we obtain the following lemma.

\begin{lemma}\label{omega(p)-property}
  The function $\omega(d)$ is multiplicative and satisfies
\begin{equation}\label{Omega-condition}
   0\leqslant\omega(p)<p,\qquad \omega(p)=1+O(p^{-1}).
\end{equation}
\end{lemma}

\section{Proof of Theorem \ref{Theorem}}

In this section, let $f(s)$ and $F(s)$ denote the classical functions in the linear sieve theory. Then by (2.8) and (2.9) of Chapter 8 in \cite{Halberstam-Richert}, we have
\begin{equation*}
  F(s)=\frac{2e^\gamma}{s},\quad 1\leqslant s\leqslant3;\qquad f(s)=\frac{2e^\gamma\log(s-1)}{s},\quad 2\leqslant s\leqslant4.
\end{equation*}
In the proof of Theorem \ref{Theorem}, let $\lambda^{\pm}(d)$ be the lower and upper bounds for Rosser's weights of level $D$, hence
for any positive integer $d$ we have
\begin{equation*}
  |\lambda^{\pm}(d)|\leqslant1,\quad \lambda^{\pm}(d)=0 \quad \textrm{if} \quad d>D \quad \textrm{or}\quad \mu(d)=0.
\end{equation*}
For further properties of Rosser's weights we refer to Iwaniec \cite{Iwaniec-1}.
Let
\begin{equation*}
  \mathscr{V}(z)=\prod_{2<p<z}\bigg(1-\frac{\omega(p)}{p}\bigg).
\end{equation*}
Then from Lemma \ref{omega(p)-property} and Mertens' prime number theorem (See \cite{Mertens}) we obtain
\begin{equation}
    \mathscr{V}(z)\asymp \frac{1}{\log N}.
\end{equation}

In order to prove Theorem \ref{Theorem}, we need the following lemma:
\begin{lemma}
  Under the condition (\ref{Omega-condition}), then if $z\leqslant D$, there holds
\begin{equation}
   \sum_{d|\mathfrak{P}}\frac{\lambda^-(d)\omega(d)}{d}\geqslant\mathscr{V}(z)\bigg(f\bigg(\frac{\log D}{\log z}\bigg)+O\big(\log^{-1/3}D\big)\bigg),
\end{equation}
and if $z\leqslant D^{1/2}$, there holds
\begin{equation}
   \sum_{d|\mathfrak{P}}\frac{\lambda^+(d)\omega(d)}{d}\leqslant\mathscr{V}(z)\bigg(F\bigg(\frac{\log D}{\log z}\bigg)+O\big(\log^{-1/3}D\big)\bigg).
\end{equation}
\end{lemma}
\begin{proof}
   See (12) and (13) of Lemma 3 in Iwaniec \cite{Iwaniec-2}.
\end{proof}

Let $M(b)=[\frac{72b}{36-b}]$. From the definition of $\mathscr{M}_r$, we know that $r\leqslant M(b)$. Therefore, we have
\begin{eqnarray}\label{R_b(N)-lower-bound}
                    \mathcal{R}_b(N)
     & \geqslant &  \sum_{\substack{m^2+p_1^3+p_2^3+p_3^3+p_4^3+p_5^4+p_6^b=N\\ m\in\mathcal{L},\qquad
                   (m,\mathfrak{P})=1\\ U_3<p_1,\,p_2\leqslant2U_3,\, U_3^*<p_3,\,p_4\leqslant 2U_3^*\\ U_4<p_5\leqslant2U_4,\,U_b<p_6\leqslant2U_b }}1-
                    \sum_{r=r(b)+1}^{M(b)} \sum_{\substack{m^2+p_1^3+p_2^3+p_3^3+p_4^3+p_5^4+p_6^b=N\\ m\in\mathscr{M}_r,\,\,
                    U_4<p_5\leqslant 2U_4,\,U_b<p_6\leqslant2U_b \\ U_3<p_1,\,p_2\leqslant2U_3,\, U_3^*<p_3,\,p_4\leqslant2U_3^* }}1
                           \nonumber \\
     & =: & \Gamma_0-\sum_{r=r(b)+1}^{M(b)}\Gamma_{r}.
\end{eqnarray}
By the property of Rosser's weight $\lambda^-(d)$ and Proposition \ref{mean-value-theorem-1}, we get
\begin{eqnarray}\label{Gamma_0-lower-bound}
  \Gamma_0  & \geqslant & \frac{1}{\log\mathbf{U}}\sum_{\substack{m^2+p_1^3+p_2^3+p_3^3+p_4^3+p_5^4+p_6^b=N\\ m\in\mathcal{L},\qquad
                   (m,\mathfrak{P})=1\\ U_3<p_1,\,p_2\leqslant2U_3,\, U_3^*<p_3,\,p_4\leqslant 2U_3^*\\ U_4<p_5\leqslant2U_4,\,U_b<p_6\leqslant2U_b }}
                          \prod_{j=1}^6\log p_j
                               \nonumber  \\
  & = & \frac{1}{\log\mathbf{U}}\sum_{\substack{m^2+p_1^3+p_2^3+p_3^3+p_4^3+p_5^4+p_6^b=N\\ m\in\mathcal{L},\,\,
                    U_4<p_5\leqslant 2U_4,\,U_b<p_6\leqslant2U_b \\ U_3<p_1,\,p_2\leqslant2U_3,\, U_3^*<p_3,\,p_4\leqslant2U_3^*}}
            \bigg(\prod_{j=1}^6\log p_j\bigg) \sum_{d|(m,\mathfrak{P})}\mu(d)
                                 \nonumber  \\
  & \geqslant & \frac{1}{\log\mathbf{U}} \sum_{\substack{m^2+p_1^3+p_2^3+p_3^3+p_4^3+p_5^4+p_6^b=N\\ m\in\mathcal{L},\,\,
                    U_4<p_5\leqslant 2U_4,\,U_b<p_6\leqslant2U_b \\ U_3<p_1,\,p_2\leqslant2U_3,\, U_3^*<p_3,\,p_4\leqslant2U_3^*}}
             \bigg(\prod_{j=1}^6\log p_j\bigg) \sum_{d|(m,\mathfrak{P})}\lambda^-(d)
                                  \nonumber  \\
  & = & \frac{1}{\log\mathbf{U}} \sum_{d|\mathfrak{P}}\lambda^-(d)J(N,d)
                                  \nonumber  \\
  & = & \frac{1}{\log\mathbf{U}} \sum_{d|\mathfrak{P}}\frac{\lambda^-(d)\mathfrak{S}_d(N)}{d} \mathcal{J}(N)+ O\big(N^{\frac{35}{36}+\frac{1}{b}}\log^{-A}N\big)
                                  \nonumber  \\
  & = & \frac{1}{\log\mathbf{U}} \bigg(\sum_{d|\mathfrak{P}}\frac{\lambda^-(d)\omega(d)}{d}\bigg)\mathfrak{S}(N)\mathcal{J}(N)
             + O\big(N^{\frac{35}{36}+\frac{1}{b}}\log^{-A}N\big)
                                  \nonumber  \\
  & \geqslant & \frac{\mathfrak{S}(N)\mathcal{J}(N)\mathscr{V}(z)}{\log\mathbf{U}} f(3)\Big(1+O\big(\log^{-1/3}D\big)\Big)
                   + O\big(N^{\frac{35}{36}+\frac{1}{b}}\log^{-A}N\big) .
\end{eqnarray}
By the property of Rosser's weight $\lambda^+(d)$ and Proposition \ref{mean-value-theorem-2}, we have
\begin{align*}
   \Gamma_r   \leqslant & \,\, \sum_{\substack{ (\ell p)^2+m^3+p_2^3+p_3^3+p_4^3+p_5^4+p_6^b=N\\ \ell\in\mathscr{N}_r,\,\,\,\,
                                            \ell p\in\mathcal{L},\,\,\,(m,\mathfrak{P})=1 \\ U_3<p_2\leqslant2U_3,\, U_3^*<p_3,\,p_4\leqslant 2U_3^*\\ U_4<p_5\leqslant2U_4,\,U_b<p_6\leqslant2U_b }} 1
                        \nonumber  \\
\end{align*}
\begin{align}\label{Gamma_r-upper-bound}
    \leqslant & \,\,\frac{1}{\log\mathbf{W}}\sum_{\substack{(\ell p)^2+m^3+p_2^3+p_3^3+p_4^3+p_5^4+p_6^b=N\\ \ell\in\mathscr{N}_r,\,\,\,\,
                                            \ell p\in\mathcal{L},\,\,\,(m,\mathfrak{P})=1 \\ U_3<p_2\leqslant2U_3,\, U_3^*<p_3,\,p_4\leqslant 2U_3^*\\ U_4<p_5\leqslant2U_4,\,U_b<p_6\leqslant2U_b}}
                  \frac{\log p}{\log\frac{U_2}{\ell}} \prod_{j=2}^6\log p_j
                        \nonumber  \\
    = & \,\,\frac{1}{\log\mathbf{W}} \sum_{\substack{(\ell p)^2+m^3+p_2^3+p_3^3+p_4^3+p_5^4+p_6^b=N\\ \ell\in\mathscr{N}_r,\quad
                                            \ell p\in\mathcal{L}, \\ U_3<p_2\leqslant2U_3,\, U_3^*<p_3,\,p_4\leqslant 2U_3^*\\
                                              U_4<p_5\leqslant2U_4,\,U_b<p_6\leqslant2U_b }}
                  \Bigg(\frac{\log p}{\log\frac{U_2}{\ell}} \prod_{j=2}^6\log p_j\Bigg)\sum_{d|(m,\mathfrak{P})}\mu(d)
                          \nonumber  \\
    \leqslant & \,\,\frac{1}{\log\mathbf{W}} \sum_{\substack{(\ell p)^2+m^3+p_2^3+p_3^3+p_4^3+p_5^4+p_6^b=N\\ \ell\in\mathscr{N}_r,\quad
                                            \ell p\in\mathcal{L}, \\ U_3<p_2\leqslant2U_3,\, U_3^*<p_3,\,p_4\leqslant 2U_3^*\\
                                              U_4<p_5\leqslant2U_4,\,U_b<p_6\leqslant2U_b }}
                  \Bigg(\frac{\log p}{\log\frac{U_2}{\ell}} \prod_{j=2}^6\log p_j\Bigg)\sum_{d|(m,\mathfrak{P})}\lambda^+(d)
                           \nonumber  \\
    = & \,\,\frac{1}{\log\mathbf{W}} \sum_{d|\mathfrak{P}} \lambda^+(d)J_r(N,d)
                            \nonumber  \\
    = & \,\,\frac{1}{\log\mathbf{W}} \sum_{d|\mathfrak{P}} \frac{\lambda^+(d)c_r(b)\mathfrak{S}_d(N)}{d\log U_2}\mathcal{J}(N)
            +O\big(N^{\frac{35}{36}+\frac{1}{b}}\log^{-A}N\big)
                            \nonumber  \\
    = & \,\,\frac{c_r(b)\mathfrak{S}(N)\mathcal{J}(N)}{(\log U_2)\log\mathbf{W}}
         \sum_{d|\mathfrak{P}}\frac{\lambda^+(d)\omega(d)}{d}+O\big(N^{\frac{35}{36}+\frac{1}{b}}\log^{-A}N\big)
                            \nonumber  \\
    \leqslant & \,\,\frac{c_r(b)\mathfrak{S}(N)\mathcal{J}(N)\mathscr{V}(z)}{\log\mathbf{U}} F(3)\Big(1+O\big(\log^{-1/3}D\big)\Big)
                   + O\big(N^{\frac{35}{36}+\frac{1}{b}}\log^{-A}N\big).
\end{align}
Define
\begin{equation*}
 C(b)=\sum_{r=r(b)+1}^{M(b)}c_r(b).
\end{equation*}
According to simple numerical calculation, we obtain
\begin{equation}\label{numerical-calculation-1}
   C(12)<0.681372,\,\, C(13)<0.430703,\,\, C(14)<0.408611,\,\, C(15)<0.649606,
\end{equation}
\begin{equation}\label{numerical-calculation-2}
   C(16)<0.496677,\,\, C(17)<0.386493,\,\, C(18)<0.621141,\,\, C(19)<0.651975,
\end{equation}
\begin{equation}\label{numerical-calculation-3}
   C(20)<0.382485,\,\, C(21)<0.631281,\,\, C(22)<0.599447,\,\, C(23)<0.426621,
\end{equation}
\begin{equation}\label{numerical-calculation-4}
   C(24)<0.394069,\,\, C(25)<0.644773,\,\, C(26)<0.603438,\,\, C(27)<0.510736,
\end{equation}
\begin{equation}\label{numerical-calculation-5}
   C(28)<0.615415,\,\, C(29)<0.502098,\,\, C(30)<0.660826,\,\, C(31)<0.403155,
\end{equation}
\begin{equation}\label{numerical-calculation-6}
   C(32)<0.656868,\,\, C(33)<0.635545,\,\, C(34)<0.669316,\,\, C(35)<0.547965.
\end{equation}
From (\ref{R_b(N)-lower-bound})--(\ref{numerical-calculation-6}), we deduce that
\begin{align*}
   \mathcal{R}_b(N) & \geqslant \big(f(3)-F(3)C(b)\big)\Big(1+O\big(\log^{-\frac{1}{3}} D\big)\Big)
                     \frac{\mathfrak{S}(N)\mathcal{J}(N)\mathscr{V}(z)}{\log\mathbf{U}} +
                                     O\bigg(\frac{N^{\frac{35}{36}+\frac{1}{b}}}{\log^{A}N}\bigg)  \\
\end{align*}
\begin{align*}
   & \geqslant \frac{2e^\gamma}{3}\big(\log2-0.681372\big)
                 \bigg(1+O\Big(\frac{1}{\log^{1/3} D}\Big)\bigg)\frac{\mathfrak{S}(N)\mathcal{J}(N)\mathscr{V}(z)}{\log\textbf{U}}
              + O\bigg(\frac{N^{\frac{35}{36}+\frac{1}{b}}}{\log^{A}N}\bigg)
                                    \nonumber  \\
   & \gg N^{\frac{35}{36}+\frac{1}{b}}\log^{-7}N,
\end{align*}
which completes the proof of Theorem \ref{Theorem}.

\section*{Acknowledgement}

   The authors would like to express the most sincere gratitude to Professor Wenguang Zhai for his valuable advice and constant encouragement.

\end{document}